\documentclass[11pt,reqno,a4paper,oneside]{amsart}

\usepackage{latexsym}
\usepackage{amsmath}
\usepackage{amsthm}
\usepackage{amssymb}
\usepackage{amsfonts}
\usepackage{comment}

\usepackage{mathrsfs}

\newcommand{\mR}{\mathbb{R}}                    
\newcommand{\abs}[1]{\lvert #1 \rvert}          
\newcommand{\norm}[1]{\lVert #1 \rVert}         

\newcommand{\re}{\mathrm{Re}}
\newcommand{\im}{\mathrm{Im}}
\newcommand{\supp}{\mathrm{supp}}

\newcommand{\mH}{\mathscr{H}}

\newcommand{\mDp}{\mathscr{D}'}

\newcommand{\id}{\mathrm{Id}}
\newcommand{\eps}{\varepsilon}

\newcommand{\closure}[1]{\overline{#1}}

\theoremstyle{definition}
\newtheorem{thm}{Theorem}[section]

\newtheorem{prop}[thm]{Proposition}
\newtheorem{cor}[thm]{Corollary}
\newtheorem{lemma}[thm]{Lemma}
\newtheorem*{definition}{Definition}

\numberwithin{equation}{section}

\title[THE ATTENUATED RAY TRANSFROM]{The attenuated ray transform on simple surfaces}

\author{Mikko Salo}
\address{Department of Mathematics and Statistics, University of Helsinki}
\email{mikko.salo@helsinki.fi}

\author{Gunther Uhlmann}
\address{Department of Mathematics, University of Washington}
\email{gunther@math.washington.edu}

\begin{document}

\begin{abstract}
We show that the attenuated geodesic ray transform on two dimensional simple surfaces is injective. Moreover we give a stability estimate and develop a reconstruction procedure.
\end{abstract}

\maketitle

\section{Introduction}

The geodesic ray transform, that is, the integration of a function along geodesics, arises as the linearization of the problem of determining a conformal factor of a Riemannian metric on a compact Riemannian manifold with boundary from the boundary distance function. This is the boundary rigidity problem, see \cite{SU} for a recent review. The standard X-ray transform, where one integrates a function along straight lines, corresponds to the case of the Euclidean metric and is the basis of medical imaging techniques such as CT and PET. The case of integration along more general geodesics arises in geophysical imaging in determining the inner structure of the Earth since the speed of elastic waves generally increases
with depth, thus curving the rays back to the Earth surface. It also arises in ultrasound imaging. Uniqueness and stability for the case of integration along geodesics on simple manifolds (see precise definition below) was shown by Mukhometov \cite{Mu} in the two dimensional case. Explicit inversion formulas in the two dimensional case were given in \cite{pestovuhlmann_imrn} for the case of constant curvature, and in the general case Fredholm type inversion formulas were given.

In this paper we consider the case of the {\sl attenuated} geodesic ray transform in two dimensions that we proceed to define.

Let $(M,g)$ be a compact 2D Riemannian manifold with boundary. The geodesics going from $\partial M$ into $M$ can be parametrized by the set $\partial_+ S(M) = \{(x,\xi) \in TM \,;\, x \in \partial M, \abs{\xi} = 1, \langle \xi,\nu \rangle \leq 0 \}$ where $\nu$ is the outer unit normal vector to $\partial M$. For any $(x,\xi) \in \partial_+ S(M)$ we let $t \mapsto \gamma(t,x,\xi)$ be the geodesic starting from $x$ in direction $\xi$. We assume that $(M,g)$ is nontrapping, which means that the time $\tau(x,\xi)$ when the geodesic $\gamma(t,x,\xi)$ exits $M$ is finite for each $(x,\xi) \in \partial_+ S(M)$.

If $a \in C^{\infty}(M)$ is the attenuation coefficient, consider the attenuated ray transform of a function $f \in C^{\infty}(M)$, 
\begin{equation*}
I^a f(x,\xi) = \int_0^{\tau(x,\xi)} f(\gamma(t,x,\xi)) \,\text{exp} \Big[ \int_0^{t} a(\gamma(s,x,\xi)) \,ds \Big] \,dt.
\end{equation*}
Here $(x,\xi) \in \partial_+ S(M)$. 

\newpage

A compact Riemannian manifold with boundary is said to be {\sl simple} if given
any two points in the boundary there is a unique minimizing geodesic joining the two points, and if the boundary is strictly convex. The notion of simplicity arises naturally in the context of the boundary rigidity problem \cite{Mi}.

Our first result shows that the attenuated ray transform on simple surfaces is injective for any attenuation coefficient.

\begin{thm} \label{thm:injectivity_functions}
Let $(M,g)$ be a simple 2D manifold, and let $a$ be any smooth complex function on $M$. If $f$ is a smooth complex function on $M$ such that $I^a f \equiv 0$, then $f \equiv 0$.
\end{thm}

Moreover we will give stability estimates and a reconstruction procedure to recover $f$ from its attenuated ray transform $I^a f$.

In the case where $M = \mR^2$ with the Euclidean metric, the corresponding injectivity result for the attenuated X-ray transform has been proved by different methods in Arbuzov, A.~L.~Bukhgeim and Kazantsev \cite{abk}, Novikov \cite{novikov}, Natterer \cite{natterer}, and Boman and Str\"omberg \cite{bomanstromberg}. These methods also come with inversion formulas. If $M$ is the unit disc in $\mR^2$ with Euclidean metric, a direct inversion formula was given by Kazantsev and A.~A.~Bukhgeim \cite{kb}. See Finch \cite{finch} and Kuchment \cite{kuchment_survey} for surveys of these and other developments in Euclidean space. The Euclidean attenuated X-ray transform is the basis of the medical imaging modality SPECT.

The attenuated geodesic ray transform arises in inverse transport problems with attenuation \cite{Ma}, \cite{Ma_survey}, when the index of refraction is anisotropic and represented by a Riemannian metric. It also arises in geophysics where there is attenuation of the elastic waves. Rather unexpectedly, this transform also appeared in the recent works \cite{DKSaU}, \cite{KSaU} in the context of Calder\'on's inverse conductivity problem in anisotropic media.

Although the attenuated ray transform is well understood in Euclidean space, much less is known about this transform on manifolds. Bal \cite{bal} proves injectivity and gives an inversion formula in the hyperbolic disc $\mathbb{H}^2$. Frigyik, Stefanov and Uhlmann \cite{FSU} prove injectivity when $(M,g)$ is simple and $g$ and $a$ are real analytic, or close to real analytic. Sharafutdinov proves injectivity of the attenuated ray transform on manifolds with a condition involving a modified Jacobi equation in \cite{Sh}, and the size and curvature of the manifold in \cite{Sh2}. Dos Santos Ferreira, Kenig, Salo, and Uhlmann \cite{DKSaU} prove the analog of Theorem \ref{thm:injectivity_functions} on any simple manifold if $\norm{a}_{L^{\infty}(M)}$ is small. A similar result, with a slightly different smallness condition, also follows from the general stability theory of \cite{FSU}. In these last results, the smallness condition arises since the methods involve a perturbation about the unattenuated case where $a = 0$.

The inversion results for $\mR^2$ and $\mathbb{H}^2$, which assume no smallness condition for $a$, are based on complex analysis and holomorphic functions. We will give a geometric version of these complex analysis arguments for simple surfaces, thus establishing injectivity of the ray transform for arbitrary attenuation coefficients. One of the key tools will be the commutator formula for the geodesic vector field and angular Hilbert transform, established in \cite{pestovuhlmann} in the study of the boundary rigidity problem in two dimensions.

At this point, let us give some other results which follow from the methods presented here. 
The next theorem considers the attenuated ray transform for combinations of functions and $1$-forms. If $F(x,\xi) = f(x) + \alpha(\xi)$ for some smooth function $f$ and $1$-form $\alpha$, where $\alpha(\xi) = \alpha_j(x) \xi^j$, the attenuated ray transform of $F$ is defined by 
\begin{equation*}
I^a F(x,\xi) = \int_0^{\tau(x,\xi)} F(\gamma(t,x,\xi),\dot{\gamma}(t,x,\xi)) \,\text{exp} \Big[ \int_0^{t} a(\gamma(s,x,\xi)) \,ds \Big] \,dt
\end{equation*}
where $(x,\xi) \in \partial_+ S(M)$. It is easy to see that this transform has nontrivial kernel since $I^a(ap + dp(\xi)) = 0$ for any $p \in C^{\infty}(M)$ with $p|_{\partial M} = 0$. The injectivity result, which also extends the corresponding result for functions, states that these are the only elements in the kernel.

\begin{thm} \label{thm:injectivity_oneforms}
Let $(M,g)$ be a simple 2D manifold and let $a \in C^{\infty}(M)$ be a complex function. Suppose that $f$ is a smooth function and $\alpha$ is a smooth $1$-form on $M$, and let $F(x,\xi) = f(x) + \alpha_j(x) \xi^j$. If $I^a F \equiv 0$, then $F = ap + dp(\xi)$ for some function $p \in C^{\infty}(M)$ with $p|_{\partial M} = 0$.
\end{thm}

Note in particular that if $f = 0$ and $a$ is nonvanishing, then any $1$-form is uniquely determined by its attenuated ray transform. Results of this type were given in the unit disc in $\mR^2$ in \cite{kb}, for simple manifolds with $\norm{a}_{L^{\infty}}$ small in \cite{DKSaU}, and for simple manifolds with $g$ and $a$ close to real analytic in \cite{HS}. For inversion formulas in $\mR^2$ see also \cite{bal_fullpartial}, \cite{natterer_vectorial}.

Once injectivity of $I^a$ is known, the general principle that the normal operator $N^a = (I^a)^* I^a$ is an elliptic pseudodifferential operator and the arguments in \cite{FSU}, \cite{HS} yield a stability result. To state this result properly, we use the solenoidal decomposition of a smooth $1$-form $\alpha$ in $M$, 
\begin{equation*}
\alpha = \alpha^s + dp,
\end{equation*}
where $\alpha^s$ is solenoidal (meaning that $\delta \alpha^s = 0$), and $p|_{\partial M} = 0$. Here $\delta$ is the codifferential. This decomposition is uniquely determined by taking $p = G(\delta \alpha)$ where $G$ is the inverse of the Dirichlet Laplacian on $M$. We also choose a simple manifold $(M_1,g)$ which is slightly larger than $(M,g)$ and extend smooth functions and $1$-forms in $M$ by zero to $M_1$. In this way $N^a$ can be viewed as a pseudodifferential operator acting on functions in $M_1$. The stability result is as follows.

\begin{thm} \label{thm:stability}
Let $(M,g)$ be a simple 2D manifold and let $a \in C^{\infty}(M)$ be a complex function. Suppose that $f$ is a smooth function and $\alpha$ is a smooth $1$-form in $M$. Then 
\begin{equation*}
\norm{f-aG(\delta \alpha)}_{L^2(M)} + \norm{\alpha^s}_{L^2(M)} \leq C \norm{N^a (f + \alpha_j \xi^j)}_{H^1(M_1)}.
\end{equation*}
\end{thm}

Finally, we outline a reconstruction procedure to determine a function $f$ from $I^a f$. For simplicity we will assume that $f$ is compactly supported in $M^{\text{int}}$ and all quantities are real valued (the complex valued case is discussed in Section \ref{sec:reconstruction}). The reconstruction procedure consists of several steps, and we refer to Section \ref{sec:preliminaries} for a more precise explanation of the notations used in the result.

\begin{thm} \label{thm:reconstruction}
Let $(M,g)$ be a simple 2D manifold, and let $a \in C^{\infty}(M)$ be real valued. A real valued function $f \in C^{\infty}_c(M^{\text{int}})$ can be determined from the knowledge of $I^a f$ using the following procedure:

\begin{enumerate}

\item[1.] 
Define a function $d$ on $\partial S(M)$ by 
\begin{equation*}
d(x,\xi) = \left\{ \begin{array}{cl} I^a f(x,\xi), & \quad (x,\xi) \in \partial_+ S(M), \\ 0, & \quad \text{otherwise}. \end{array} \right.
\end{equation*}

\item[2.] 
Find an odd holomorphic function $w$ such that $\mH w = -a$.

\item[3.] 
Let $\beta = (\id-iH)(e^{-w} d)$ on $\partial S(M)$.

\item[4.] 
Let $v = \beta \circ \psi + u^{(I^0)^{-1}(A_-^* \beta)}$ in $SM$, where $A_-^* \beta = \beta - \beta \circ \psi$ on $\partial_+ S(M)$ and $(I^0)^{-1}$ is the inverse of the geodesic ray transform in $(M,g)$, in the sense that 
\begin{equation*}
(I^0)^{-1} I^0 (\phi + \alpha_j \xi^j) = \phi + \alpha_j \xi^j
\end{equation*}
for a smooth function $\phi$ and a solenoidal $1$-form $\alpha$.

\item[5.] 
Define $\hat{m} = \frac{1}{2} \re[(\id-iH)(e^w v)]$ and $\hat{u} = \hat{m} - \hat{m}_0$.

\item[6.] 
Define $q = (d-\hat{u})_0 \circ \psi + (u^{(\mH \hat{u} + a \hat{u})_-})_0$, and let $u = q + \hat{u}$.

\item[7.] 
Let $f = -(\mH u + au)_0$.

\end{enumerate}

\end{thm}

There are two nontrivial steps (Steps 2 and 4) in the above result: they require to find a holomorphic integrating factor $w$ to the transport equation $(\mH + a)u = 0$, and to invert the geodesic ray transform $I^0$ with zero attenuation. Both these steps can be achieved in an explicit way if $(M,g)$ has constant curvature, or if $(M,g)$ is a small perturbation of a constant curvature manifold (see Section \ref{sec:reconstruction}). However, it is not clear how to carry out these steps explicitly in a general simple 2D manifold.

It seems that even when $M$ is a domain in $\mR^2$ with Euclidean metric, the reconstruction procedure does not reduce to a simple formula such as in \cite{kb}, \cite{novikov}. It would be interesting to give a reconstruction procedure which would reduce to such a simple formula on constant curvature manifolds.

The structure of the paper is as follows. Section \ref{sec:preliminaries} establishes notation and preliminaries related to geodesic flow, Hilbert transform, and functions which are holomorphic in the angular variable. In Section \ref{sec:strategy} we explain the strategy of the injectivity proof, starting with a simple inversion scheme based on holomorphic solutions of the transport equation, and discussing the modifications to this scheme required in the attenuated case. The first main step in the proof, the construction of holomorphic integrating factors, is achieved in Section \ref{sec:integrating_factors} using pseudodifferential arguments. The second main step consists in proving that solutions of certain transport equations are necessarily holomorphic. This is done in Section \ref{sec:uniqueness}, where also Theorems \ref{thm:injectivity_functions} to \ref{thm:stability} are proved. The final Section \ref{sec:reconstruction} gives a reconstruction procedure and proves Theorem \ref{thm:reconstruction}.

\subsection*{Acknowledgements}
M.S. is partly supported by the Academy of Finland. G.U. is supported in part by NSF and a Walker Family Endowed Professorship.

\section{Preliminaries} \label{sec:preliminaries}

We refer to \cite{pestovuhlmann_imrn}, \cite{pestovuhlmann}, \cite{Sh} for the following facts. Assume that $(M,g)$ is a compact 2D Riemannian manifold with boundary $\partial M$. We will also assume that $(M,g)$ is nontrapping and $\partial M$ is strictly convex (see below). We denote the inner product on tangent vectors and other tensors by $\langle \,\cdot\,, \,\cdot\, \rangle$ and the corresponding norm by $\abs{\,\cdot\,}$.

\subsection{Geodesics}

We will mostly work on the unit sphere bundle given by 
\begin{equation*}
SM = \bigcup_{x \in M} S_x, \qquad S_x = \{ (x,\xi) \in TM \,;\, \abs{\xi} = 1 \}.
\end{equation*}
The manifold $SM$ has boundary $\partial S(M) = \{ (x,\xi) \in SM \,;\, x \in \partial M \}$. The outer unit normal vector of $\partial M$ is denoted by $\nu$, and the sets of inner and outer vectors on $\partial M$ are given by 
\begin{equation*}
\partial_{\pm} S(M) = \{ (x,\xi) \in SM \,;\, x \in \partial M, \ \pm \langle \xi, \nu \rangle \leq 0 \}.
\end{equation*}

If $(x,\xi)$ is a point in $SM$ we denote by $\gamma(t,x,\xi)$ the geodesic on $M$ satisfying $\gamma(0,x,\xi) = x$ and $\dot{\gamma}(0,x,\xi) = \xi$. The geodesic flow is the map 
\begin{equation*}
\varphi_t: SM \to SM, \ \varphi_t(x,\xi) = (\gamma(t,x,\xi), \dot{\gamma}(t,x,\xi))
\end{equation*}
if $t$ is such that the right hand side is well defined. The nonnegative time when a geodesic $\gamma(\,\cdot\,,x,\xi)$ exits $M$ is denoted by $\tau(x,\xi)$. The manifold $(M,g)$ is said to be nontrapping if $\tau(x,\xi)$ is finite for any $(x,\xi) \in SM$. The boundary $\partial M$ is said to be strictly convex if its second fundamental form is positive definite.

Since $(M,g)$ is nontrapping and has strictly convex boundary, the next result holds by \cite[Section 4.1]{Sh}. 

\begin{lemma} \label{tau_regularity}
$\tau$ is continuous in $SM$ and smooth in $SM \smallsetminus S(\partial M)$, and further the function $\tau_-: \partial S(M) \to \mR$ defined by 
\begin{equation*}
\tau_-(x,\xi) = \left\{ \begin{array}{cl} \frac{1}{2} \tau(x,\xi), & (x,\xi) \in \partial_+ S(M), \\ -\frac{1}{2} \tau(x,-\xi), & (x,\xi) \in \partial_- S(M) \end{array} \right.
\end{equation*}
is smooth.
\end{lemma}

\subsection{Scattering relation}

The scattering relation $\alpha$ maps an inner unit vector $(x,\xi) \in \partial_+ S(M)$ to the outer vector $\varphi_{\tau(x,\xi)}(x,\xi)$. Thus, $\alpha$ takes the starting point on the boundary and direction of a geodesic and gives out the endpoint and direction of that geodesic. It is possible to define $\alpha$ as a smooth map on all of $\partial S(M)$ by 
\begin{equation*}
\alpha(x,\xi) = \varphi_{2\tau_-(x,\xi)}(x,\xi), \quad (x,\xi) \in \partial S(M).
\end{equation*}
Then $\alpha$ is a diffeomorphism $\partial S(M) \to \partial S(M)$ and $\alpha^2 = \id$.

\subsection{Geodesic vector field}

The geodesic vector field $\mH$ is the vector field on $SM$ which acts on smooth functions $u$ on $SM$ by 
\begin{equation*}
\mH u(x,\xi) = \frac{\partial}{\partial t} u(\varphi_t(x,\xi))\Big|_{t=0}.
\end{equation*}
We consider two boundary problems related to $\mH$. If $F$ is a smooth function on $SM$, then the problem 
\begin{equation*}
\mH u = -F \text{ in $SM$}, \quad u|_{\partial_- S(M)} = 0
\end{equation*}
has the solution $u = u^F$ where 
\begin{equation*}
u^F(x,\xi) = \int_0^{\tau(x,\xi)} F(\varphi_t(x,\xi)) \,dt.
\end{equation*}
If $w$ is a smooth function on $\partial_+ S(M)$ then the problem 
\begin{equation*}
\mH u = 0 \text{ in $SM$}, \quad u|_{\partial_+ S(M)} = w
\end{equation*}
has the solution $u = w_{\psi}$ given by 
\begin{equation} \label{w_def}
w_{\psi} = w \circ \alpha \circ \psi
\end{equation}
where $\psi$ is the end point map $\psi(x,\xi) = \varphi_{\tau(x,\xi)}(x,\xi)$, and $\alpha$ is the scattering relation.

Since $\tau$ is continuous on $SM$ and smooth on $SM \smallsetminus S(\partial M)$, the same is true for $u^F$ and $w_{\psi}$. It is a minor inconvenience that these functions are not smooth on $SM$ in general. The space of those $w$ for which $w_{\psi}$ is smooth in $SM$ is denoted by 
\begin{equation*}
C^{\infty}_{\alpha}(\partial_+ S(M)) = \{ w \in C^{\infty}(\partial_+ S(M)) \,;\, w_{\psi} \in C^{\infty}(SM) \}.
\end{equation*}
This space was characterized in \cite{pestovuhlmann} in terms of the operator $A_+$ of even continuation with respect to $\alpha$, acting on $w \in C^{\infty}(\partial_+ S(M))$ by 
\begin{equation*}
A_+ w(x,\xi) = \left\{ \begin{array}{cl} w(x,\xi), & (x,\xi) \in \partial_+ S(M), \\ w(\alpha(x,\xi)), & (x,\xi) \in \partial_- S(M). \end{array} \right.
\end{equation*}

\begin{lemma} \label{wpsi_regularity}
$C^{\infty}_{\alpha}(\partial_+ S(M)) = \{ w \in C^{\infty}(\partial_+ S(M)) \,;\, A_+ w \in C^{\infty}(\partial S(M)) \}$.
\end{lemma}

As for $u^F$, sometimes we can work under the extra assumption that $F$ vanishes near $\partial M$ in which case $u^F$ is smooth on $SM$. At other times, we can use the fact that the odd part of $u^F$ is smooth in $SM$ provided that $F$ is even. If $u$ is a function on $SM$ the even and odd parts are defined by 
\begin{align*}
u_{\pm}(x,\xi) &= \frac{1}{2}(u(x,\xi) \pm u(x,-\xi)).
\end{align*}
Of course, $u$ is called even (resp.~odd) if $u = u_+$ (resp.~$u = u_-$).

\begin{lemma} \label{uf_regularity}
If $F$ is an even smooth function on $SM$, then $u^F_-$ is a smooth function in $SM$ and satisfies $\mH u^F_- = -F$.
\end{lemma}
\begin{proof}
The last statement follows since $R^* \mH u = -\mH R^* u$ where $R$ is the map $R(x,\xi) = (x,-\xi)$. This implies that $(\mH u)_+ = \mH u_-$.

We will reduce the smoothness statement to Lemma \ref{wpsi_regularity}. Let $(\tilde{M},g)$ be a nontrapping manifold with strictly convex boundary so that $M \subseteq \tilde{M}^{\text{int}}$ (this can be achieved by embedding $(M,g)$ to a compact manifold $(S,g)$ without boundary and by looking at a small neighborhood of $M$ in $S$). If $\tilde{\tau}(x,\xi)$ is the exit time of geodesics in $(\tilde{M},g)$, we know that $\tilde{\tau}$ is smooth in $S(\tilde{M}^{\text{int}})$.

Extend $F$ as a smooth even function into $S \tilde{M}$, and define 
\begin{equation*}
\tilde{u}(x,\xi) = \int_0^{\tilde{\tau}(x,\xi)} F(\varphi_t(x,\xi)) \,dt
\end{equation*}
where $\varphi_t$ is the geodesic flow in $(\tilde{M},g)$. Then $\tilde{u} \in C^{\infty}(SM)$ and $\mH \tilde{u} = -F$ in $SM$.

Let $w = (\tilde{u} - u^F_-)|_{\partial_+ S(M)}$. Since $\mH (\tilde{u} - u^F_-) = 0$ in $S(M^{\text{int}})$ and $\tilde{u} - u^F_-$ is continuous in $SM$, we obtain $\tilde{u} - u^F_- = w_{\psi}$. Thus, to show that $u^F_-$ is smooth in $SM$ it is enough by Lemma \ref{wpsi_regularity} to prove that $A_+ w$ is in $C^{\infty}(\partial S(M))$. A short computation, using that $F$ is even, gives that for $(x,\xi) \in \partial S(M)$ 
\begin{equation*}
A_+ w(x,\xi) = \frac{1}{2} \int_0^{\tilde{\tau}(x,\xi)} F(\varphi_t(x,\xi)) \,dt + \frac{1}{2} \int_{2\tau_-(x,\xi)}^{\tilde{\tau}(x,\xi)} F(\varphi_t(x,\xi)) \,dt.
\end{equation*}
We know that $2\tau_-$ is smooth in $\partial S(M)$ by Lemma \ref{tau_regularity}, hence also $A_+ w$ is smooth.
\end{proof}

\subsection{Hilbert transform}

To discuss functions which are (anti)holomorphic in the angular variable, we introduce the fiberwise Hilbert transform which acts on smooth functions on $SM$ by 
\begin{equation*}
Hu(x,\xi) = \frac{1}{2\pi} \int_{S_x} \frac{1+\langle \xi,\eta \rangle}{\langle \xi_{\perp}, \eta \rangle} u(x,\eta) \,dS_x(\eta), \quad (x,\xi) \in S M.
\end{equation*}
The integral is understood as a principal value. Here $(\xi_{\perp})_j = \eps_{jk} \xi^k$ where $\eps$ is the clockwise rotation by 90 degrees: 
\begin{equation*}
\eps = \sqrt{\det\,g} \left( \begin{array}{cc} 0 & 1 \\ -1 & 0 \end{array} \right).
\end{equation*}
If $H_0$ is the usual Hilbert transform on the unit circle $S^1$, and if $F_x$ is any orientation preserving isometry from $S_x$ onto $S^1$ (such a map is unique up to rotation on $S^1$), one has for fixed $x$ 
\begin{equation} \label{hilbert_circle}
H = F_x^* H_0 (F_x^{-1})^*.
\end{equation}
The last identity allows to transfer standard properties of the Hilbert transform on the unit circle to the present setting (see also \cite[Section 8]{Sh_deformation}).

A crucial ingredient for our arguments is a commutator formula proved in \cite{pestovuhlmann}, which gives a connection between the geodesic vector field and the fiberwise Hilbert transform.

\begin{prop}
If $u$ is a smooth function on $SM$ then 
\begin{equation} \label{commutator_formula}
[H,\mH]u = \mH_{\perp} u_0 + (\mH_{\perp} u)_0.
\end{equation}
Here $u_0$ is the average of $u$ over the angular variable:
\begin{equation*}
u_0(x) = \frac{1}{2\pi} \int_{S_x} u(x,\xi) \,dS_x(\xi), \quad x \in M.
\end{equation*}
\end{prop}

We have used the vector field $\mH_{\perp} = (\xi_{\perp})^j \nabla_j$ on $SM$. Here $\nabla$ is the horizontal derivative on $SM$ \cite{Sh}. In local coordinates it is given by 
\begin{equation*}
\nabla_j u(x,\xi) = \frac{\partial}{\partial x^j} (u(x,\xi/\abs{\xi})) - \Gamma_{jk}^l \xi^k \frac{\partial}{\partial \xi^l}(u(x,\xi/\abs{\xi})).
\end{equation*}

We collect some further basic properties of the Hilbert transform \cite{pestovuhlmann}. These involve even and odd functions with respect to the angular variable. 

\begin{prop}
The Hilbert transform maps even (resp. odd) functions with respect to $\xi$ to even (resp. odd) functions. If $u$ is a function on $SM$ then $H u_{\pm} = H_{\pm} u$ where 
\begin{align*}
H_+ u(x,\xi) &= \frac{1}{2\pi} \int_{S_x} \frac{\langle \xi, \eta \rangle}{\langle \xi_{\perp}, \eta \rangle} u(x,\eta) \,dS_x(\eta), \\
H_- u(x,\xi) &= \frac{1}{2\pi} \int_{S_x} \frac{1}{\langle \xi_{\perp}, \eta \rangle} u(x,\eta) \,dS_x(\eta).
\end{align*}
Also, if $u$ is a function on $SM$ then $(Hu)_0 = 0$, and if $u = u(x)$ then $Hu \equiv 0$.
\end{prop}

\subsection{Holomorphic functions}

The arguments below will be based on the ability of finding (anti)holomorphic solutions to transport equations. Here, (anti)holomorphic refers to the angular variable. The precise definition uses the Hilbert transform and is as follows.

\begin{definition}
A function $u$ on $SM$ is called holomorphic if 
\begin{equation*}
(\id - iH)u = u_0.
\end{equation*}
We say that $u$ is antiholomorphic if 
\begin{equation*}
(\id + iH)u = u_0.
\end{equation*}
\end{definition}

The next result, which follows by \eqref{hilbert_circle}, will be used many times below.

\begin{lemma}
The product of two (anti)holo\-morphic functions is (anti)holo\-morphic, and $e^w$ is (anti)holomorphic if $w$ is (anti)holomorphic.
\end{lemma}

As an example, and to obtain some intuition into the arguments below, we will discuss the above notions in the case where $M$ is an open set in $\mR^2$ with Euclidean metric. Then $SM = M \times S^1$, and any function $u$ on $SM$ may be written as Fourier series 
\begin{equation*}
u(x,e^{i\theta}) = \sum_{k=-\infty}^{\infty} u_k(x) e^{ik\theta}.
\end{equation*}
Here $u_k(x)$ are the Fourier coefficients 
\begin{equation*}
u_k(x) = \frac{1}{2\pi} \int_0^{2\pi} e^{-ik\theta} u(x,e^{i\theta}) \,d\theta.
\end{equation*}
Then the even and odd parts of $u$ are obtained by just taking the even or odd Fourier coefficients, 
\begin{align*}
u_+(x,e^{i\theta}) &= \sum_{k \text{ even}} u_k(x) e^{ik\theta}, \\
u_-(x,e^{i\theta}) &= \sum_{k \text{ odd}} u_k(x) e^{ik\theta}.
\end{align*}
Also, with the convention $\text{sgn}(0) = 0$, 
\begin{equation*}
H(e^{ik\theta}) = -\text{sgn}(k) i e^{ik\theta}.
\end{equation*}
Therefore 
\begin{align*}
(\id+iH) u &= u_0(x) + 2 \sum_{k=1}^{\infty} u_k(x) e^{ik\theta}, \\
(\id-iH) u &= u_0(x) + 2 \sum_{k=-\infty}^{-1} u_k(x) e^{ik\theta}.
\end{align*}
Now $(\id \pm iH) u = u_0$ means that the negative or positive Fourier coefficients vanish. Thus, $u$ is holomorphic (antiholomorphic) if and only if for any $x$ in $M$, $u(x,\,\cdot\,)$ extends into a holomorphic (antiholomorphic) function in the unit disc.

\section{Strategy of proof} \label{sec:strategy}

Our proof of Theorem \ref{thm:injectivity_functions} reduces the attenuated ray transform to the analysis of solutions of a transport equation. Let $\mH$ be the geodesic vector field on the unit sphere bundle $S M$. Then $I^a f = u|_{\partial_+ S M}$ where $u$ satisfies 
\begin{equation*}
(\mH + a)u = -f \text{ in $SM$}, \quad u|_{\partial_- SM} = 0.
\end{equation*}
Holomorphic solutions of certain transport equations will be crucial in the proof. To explain why such solutions might be useful, we first discuss a simple scheme which would imply injectivity and which turns out to work in the unattenuated case. In the end of the section we outline the strategy for the attenuated case.

\subsection*{First inversion scheme}
Let $a$ and $f$ be real valued, and let $u$ be the solution given above. Motivated by the earlier result \cite{kb} in $\mR^2$, it turns out that injectivity of $I^a$ would be a consequence of the following idea:

\begin{quote}
\emph{{Produce a function $u^*$, which is holomorphic (or antiholomorphic) in the angular variable, such that $(u^*)_0 = 0$ and }} 
\begin{equation} \label{ustar_eq}
(\mH + a)u^* = -f
\end{equation}
\emph{{and such that $u^*|_{\partial S(M)}$ is determined by $u|_{\partial S(M)}$.}}
\end{quote}

To see how the above statement could be used to invert the attenuated ray transform, it is enough to take the imaginary part of \eqref{ustar_eq} to obtain 
\begin{equation*}
(\mH + a)(\im\,u^*) = 0.
\end{equation*}
Since $\mH + a = e^{u^a_-} \mH e^{-u^a_-}$ (recall that $u^a_-$ is smooth in $SM$ by Lemma \ref{uf_regularity}) this shows that $e^{-u^a_-} \im\,u^*$ is constant on geodesics, and therefore $\im\,u^*$ is determined by its boundary values on $\partial_+ S(M)$. But $u^*$ is (anti)holomorphic with zero average so its real part is determined by the imaginary part, and we obtain that $u^*$ in $SM$ is determined by $u|_{\partial_+ S(M)} = I^a f$. Then $f$ can be reconstructed from $I^a f$ for instance by taking averages over the angular variable in \eqref{ustar_eq}:
\begin{equation*}
f = -((\mH+a)u^*)_0.
\end{equation*}
This would give an inversion formula for the attenuated ray transform.

\subsection*{The unattenuated case}
The above scheme actually works in the case $a = 0$ where no attenuation is present, and results in a similar inversion formula as in \cite{pestovuhlmann_imrn}. Let $(M,g)$ be a simple surface, and assume for simplicity that $f \in C^{\infty}_c(M^{\text{int}})$ is real valued. We would like to recover $f$ from the knowledge of the geodesic ray transform $I^0 f|_{\partial_+ S(M)}$.

Let $u$ be the solution of 
\begin{equation*}
\mH u = -f \text{ in $SM$}, \quad u|_{\partial_- S(M)} = 0.
\end{equation*}
Then $u = u^f$ and $u|_{\partial_+ S(M)} = I^0 f$. To obtain a holomorphic solution with zero average, the first idea is to take 
\begin{equation*}
u^* = (\id + iH) u_-.
\end{equation*}
To see if $u^*$ solves the transport equation, we use the fact that $\mH u_- = -f$ and compute by \eqref{commutator_formula}
\begin{equation*}
\mH u^* = (\id + iH) \mH u_- - i [H,\mH] u_- = -f - i(\mH_{\perp} u)_0.
\end{equation*}
The last expression on the right was analyzed in \cite[Section 5]{pestovuhlmann_imrn}.

\begin{prop}
If $(M,g)$ is simple then the operators 
\begin{align*}
Wf &= (\mH_{\perp} u^f)_0, \\
W^* f &= (u^{\mH_{\perp} f})_0
\end{align*}
have smooth integral kernels and extend as maps from $L^2(M)$ to $C^{\infty}(M)$. Also, $W^*$ is the adjoint of $W$, and if $(M,g)$ has constant curvature then $W \equiv W^* \equiv 0$ (this last statement is also true in the presence of conjugate points).
\end{prop}

It follows that 
\begin{equation} \label{unatt_ustar_eq_first}
\mH u^* = -f - iWf.
\end{equation}
Thus, if $(M,g)$ has constant curvature, then $Wf = 0$ and $u^*$ is the required holomorphic solution with zero average. Note that 
\begin{equation*}
u^*|_{\partial S(M)} = (\id + iH) u_-|_{\partial S(M)}
\end{equation*}
so $u^*|_{\partial S(M)}$ is indeed determined by $u|_{\partial S(M)}$. The scheme above now gives an inversion formula for the geodesic ray transform $I^0$.

In the case where $(M,g)$ does not have constant curvature the quantity $Wf$ may be nonzero, but one can still obtain a Fredholm type inversion formula as in \cite{pestovuhlmann_imrn}. The right hand side in \eqref{unatt_ustar_eq_first} is complex so splitting into real and imaginary parts is not immediately useful. However, one can iterate once more and introduce the antiholomorphic odd function 
\begin{equation} \label{ustarstar}
u^{**} = (\id - iH) u^{f+iWf}_-.
\end{equation}
This satisfies by \eqref{commutator_formula} and the fact that $\mH u^{f+iWf}_- = -f-iWf$ 
\begin{align*}
\mH u^{**} &= (\id - iH)\mH u^{f+iWf}_- + i [H,\mH] u^{f+iWf}_- \\
 &= - f - iWf + i(\mH_{\perp} u^{f+iWf})_0 \\
 &= - f - W^2 f.
\end{align*}
Now $f + W^2 f$ is real and the inversion scheme above can be used to recover $f + W^2 f$ from $I^0 f$. Thus we have constructed $f$ up to a Fredholm error.

\subsection*{The attenuated case}
If $(M,g)$ is a simple 2D manifold and $a$ is a real attenuation coefficient, we will use a modification of this inversion scheme which still involves holomorphic solutions. We explain the idea for proving injectivity. Suppose that $f \in C^{\infty}_c(M^{\text{int}})$ is real valued and $I^a f \equiv 0$. Then the corresponding solution $u$ of the transport equation satisfies 
\begin{equation*}
(\mH + a)u = -f \text{ in $SM$}, \quad u|_{\partial SM} = 0.
\end{equation*}
The first step in the proof is to find a \emph{holomorphic integrating factor}: an odd holomorphic function $w \in C^{\infty}(SM)$ such that one has the operator identity 
\begin{equation*}
e^w \mH e^{-w} = \mH + a.
\end{equation*}
We derive a characterization for the existence of such $w$, and we will employ pseudodifferential arguments to show that one can always find a holomorphic integrating factor. Given such $w$, the function $e^{-w} u$ satisfies 
\begin{equation*}
\mH(e^{-w} u) = -e^{-w} f \text{ in $SM$}, \quad e^{-w} u|_{\partial SM} = 0.
\end{equation*}
The second main step is to show that any solution which satisfies such a transport equation with holomorphic right hand side and which vanishes on $\partial SM$ is necessarily holomorphic. This uses the commutator formula \eqref{commutator_formula} for $H$ and $\mH$, and boils down to the injectivity of the unattenuated ray transform on $1$-forms \cite{An}. We obtain that $e^{-w} u$ is a holomorphic function, and since $e^w$ is holomorphic then so is $u$.

This argument shows that whenever $I^a f \equiv 0$, the solution $u$ of the transport equation must be holomorphic. Since $u$ is real it is also antiholomorphic, which shows that $u$ only depends on $x$. Thus $u = u_0$, and the transport equation reads 
\begin{equation*}
d u_0(\xi) + a u_0 = -f \text{ in $SM$}, \quad u_0|_{\partial SM} = 0.
\end{equation*}
Evaluating at $\pm \xi$ shows that $du_0 = 0$, and consequently $u_0 = 0$ and $f = 0$.

\section{Holomorphic integrating factors} \label{sec:integrating_factors}

Let $(M,g)$ be a simple surface and let $a$ be a smooth function on $M$. We consider the problem of finding holomorphic integrating factors for the operator $\mH + a$. More precisely, we are looking for smooth holomorphic functions $w$ on $SM$ such that 
\begin{equation*}
(\mH + a)v = e^w \mH (e^{-w} v)
\end{equation*}
for all smooth functions $v$ on $SM$. An equivalent condition is that $w$ should satisfy 
\begin{equation*}
\mH w = -a.
\end{equation*}
The main result of this section shows that holomorphic integrating factors always exist.

\begin{prop} \label{prop:holomorphic_integratingfactor}
Let $(M,g)$ be a simple 2D manifold and let $a$ be any smooth complex function on $M$. There exist a holomorphic $w \in C^{\infty}(SM)$ and an antiholomorphic $\tilde{w} \in C^{\infty}(SM)$, both odd functions, such that $\mH w = \mH \tilde{w} = -a$.
\end{prop}

One smooth solution to the equation $\mH w = -a$ is given by $w = u^a_-$. We begin with the simple observation that on constant curvature manifolds, the projection of $u^a_-$ to holomorphic functions also satisfies this condition.

\begin{lemma}
Consider the operators $\Gamma$, $\tilde{\Gamma}$ acting on functions on $M$ by 
\begin{equation*}
\Gamma a = (\id + iH) u^a_-, \qquad \tilde{\Gamma} a = (\id-iH) u^a_-.
\end{equation*}
Then $\Gamma$ (resp.~$\tilde{\Gamma}$) maps $C^{\infty}(M)$ to odd functions in $C^{\infty}(SM)$ which are holomorphic (resp.~antiholomorphic) in the angular variable. One has 
\begin{equation*}
\mH \Gamma a = -a - iWa, \qquad \mH \tilde{\Gamma} a = -a + iWa.
\end{equation*}
If $Wa = 0$, then 
\begin{equation*}
\mH \Gamma a = \mH \tilde{\Gamma} a = -a.
\end{equation*}
\end{lemma}
\begin{proof}
The function $u^a_-$ is smooth by Lemma \ref{uf_regularity}, hence $\Gamma a$ and $\tilde{\Gamma} a$ are odd and smooth. By the commutator formula \eqref{commutator_formula} 
\begin{equation*}
\mH \Gamma a = (\id+iH) \mH u^a_- - i[H,\mH] u^a_- = -a - i(\mH_{\perp} u^a)_0 = - a - i Wa.
\end{equation*}
The computation for $\tilde{\Gamma}$ is analogous.
\end{proof}

\begin{cor} \label{cor:integratingfactors}
If $a \in \text{Ran}(\id+iW) \cap \text{Ran}(\id-iW)$, then there exist smooth holomorphic $w$ and antiholomorphic $\tilde{w}$ such that $\mH w = \mH \tilde{w} = -a$.
\end{cor}
\begin{proof}
If $a$ satisfies the given condition, there are $b, \tilde{b} \in C^{\infty}(M)$ such that $a = b+iWb = \tilde{b}-iW\tilde{b}$. Letting $w = \Gamma b$ and $\tilde{w} = \tilde{\Gamma} \tilde{b}$, we have 
\begin{equation*}
\mH w = -b - iWb = -a
\end{equation*}
and similarly for $\tilde{w}$.
\end{proof}

The corollary shows that on any manifold for which $\id+iW$ and $\id-iW$ are surjective, one can find holomorphic integrating factors. This includes constant curvature manifolds (since $W = 0$) and small perturbations of these (since $W$ has small norm \cite{krishnan}). However, for general simple manifolds we do not know if $\id \pm iW$ are surjective and one needs to work harder.

First we give a characterization of those $a$ for which $\mH w = -a$ for some (anti)holomorphic function $w$.

\begin{lemma} \label{lemma:integratingfactor_characterization}
Let $a$ be a smooth complex function on $M$. The following conditions are equivalent:
\begin{enumerate}
\item[(1)]
There exists a holomorphic (resp.~antiholomorphic) odd function $w$ in $C^{\infty}(SM)$ such that 
\begin{equation*}
\mH w = -a.
\end{equation*}
\item[(2)]
There exists a function $h \in C^{\infty}_{\alpha}(\partial_+ S(M))$ and an antiholomorphic (resp.~holomorphic) even function $b \in C^{\infty}(SM)$ so that 
\begin{equation*}
a = b_0 + i(\mH_{\perp} u^b)_0 + i(\mH_{\perp} h_{\psi})_0
\end{equation*}
(resp.~$a = b_0 - i(\mH_{\perp} u^b)_0 - i(\mH_{\perp} h_{\psi})_0$).
\end{enumerate}
\begin{proof}
Let first $w$ be a holomorphic odd function with $\mH w = -a$ (the case of antiholomorphic functions is analogous). Then $w = (\id+iH)\hat{w}$ for some odd $\hat{w} \in C^{\infty}(SM)$, in fact one may take $\hat{w} = \frac{1}{2} w$.

Now by the commutator formula \eqref{commutator_formula} 
\begin{align*}
-a &= \mH w = \mH (\id+iH) \hat{w} \\
 &= (\id+iH) \mH \hat{w} - i[H,\mH] \hat{w} \\
 &= (\id+iH) \mH \hat{w} - i (\mH_{\perp} \hat{w})_0.
\end{align*}
This shows that $(\id+iH) \mH \hat{w}$ only depends on $x$. Consequently the function $b = -\mH \hat{w}$ is antiholomorphic and even, and we have $(\id+iH) \mH \hat{w} = -b_0$. The equality $\mH \hat{w} = -b$ implies that $\mH(\hat{w}-u^b_-) = 0$, hence $\hat{w} = u^b_- + h_{\psi}$ for some $h \in C^{\infty}_{\alpha}(\partial_+ S(M))$ by Lemma \ref{uf_regularity}. Thus 
\begin{align*}
a = b_0 + i(\mH_{\perp} u^b)_0 + i (\mH_{\perp} h_{\psi})_0.
\end{align*}
We have proved (2).

Conversely, assume $a$ is of the form given in (2) with $b$ antiholomorphic and even. Define $\hat{w} = u^b_- + (h_{\psi})_-$. Then $\hat{w}$ is odd, $\mH \hat{w} = -b$, and 
\begin{equation*}
a = b_0 + i(\mH_{\perp} \hat{w})_0 = -(\id+iH)\mH \hat{w} + i(\mH_{\perp} \hat{w})_0.
\end{equation*}
Define $w = (\id+iH) \hat{w}$. It follows that $w$ is holomorphic and odd, and 
\begin{align*}
\mH w &= (\id+iH) \mH \hat{w} -i[H,\mH] \hat{w} = (\id+iH) \mH \hat{w} -i (\mH_{\perp} \hat{w})_0 \\
 &= -a
\end{align*}
as required.
\end{proof}
\end{lemma}

We will study the operator appearing in Lemma \ref{lemma:integratingfactor_characterization},
\begin{equation} \label{s_def}
S: C^{\infty}_{\alpha}(\partial_+ S(M)) \to C^{\infty}(M), \ S h = (\mH_{\perp} h_{\psi})_0.
\end{equation}
The main point is the following result.

\begin{lemma} \label{lemma:s_surjective}
The operator $S: C^{\infty}_{\alpha}(\partial_+ S(M)) \to C^{\infty}(M)$ is surjective.
\end{lemma}

Given this, it is easy to see that any attenuation coefficient admits holomorphic and antiholomorphic integrating factors.

\begin{proof}[Proof of Proposition \ref{prop:holomorphic_integratingfactor}]
Given the attenuation $a$, we can take $b$ to be any antiholomorphic even function in $C^{\infty}(SM)$ (for instance $b = 0$). Choose $h \in C^{\infty}_{\alpha}(\partial_+ S(M))$ such that $iSh = a-b_0-i(\mH_{\perp} u^b)_0$. Then 
\begin{equation*}
b_0 + i(\mH_{\perp} u^b)_0 + i(\mH_{\perp} h_{\psi})_0 = a
\end{equation*}
and Lemma \ref{lemma:integratingfactor_characterization} shows that $\mH w = -a$ for some odd holomorphic $w$. The antiholomorphic case is analogous.
\end{proof}

It remains to prove Lemma \ref{lemma:s_surjective}. We will establish the surjectivity of $S$ by proving that its adjoint is injective and has closed range. In fact, the adjoint involves the unattenuated ray transform $I^0$ on 1-forms. Recall the space $L^2_{\mu}(\partial_+ S(M))$ where $\mu(x,\xi) = -\langle \xi, \nu(x) \rangle$ and 
\begin{equation*}
(h,h')_{L^2_{\mu}(\partial_+ S(M))} = \int_{\partial_+ S(M)} h h' \mu \,d(\partial S(M)),
\end{equation*}
and $d(\partial S(M))$ is the natural volume form on $\partial S(M)$ \cite[Appendix A.4]{DPSU}.

\begin{lemma} \label{lemma:s_adjoint}
If $h \in C^{\infty}_{\alpha}(\partial_+ S(M))$, $f \in C^{\infty}_c(M^{\text{int}})$ then 
\begin{equation*}
(Sh, f)_{L^2(M)} = (h, -\frac{1}{2\pi} I^0 \mH_{\perp} f)_{L^2_{\mu}(\partial_+ S(M))}.
\end{equation*}
\end{lemma}
\begin{proof}
We claim that for any $v \in C^{\infty}(SM)$ one has 
\begin{equation} \label{hperp_integrationbyparts}
\int_{SM} (\mH_{\perp} f) v \,d(SM) = -\int_{SM} f(\mH_{\perp} v) \,d(SM).
\end{equation}
If this holds then we obtain by Santal{\'o}'s formula \cite[Appendix A.4]{DPSU} that 
\begin{align*}
 &\int_M (Sh)f \,dM = \frac{1}{2\pi} \int_{SM} (\mH_{\perp} h_{\psi}) f \,d(SM) = -\frac{1}{2\pi} \int_{SM} h_{\psi} (\mH_{\perp} f) \,d(SM) \\
 &= -\frac{1}{2\pi} \int_{\partial_+ S(M)} \int_0^{\tau(x,\xi)} h_{\psi}(\varphi_t(x,\xi)) \mH_{\perp} f(\varphi_t(x,\xi)) \mu \,dt \,d(\partial S(M)) \\
 &= -\frac{1}{2\pi} \int_{\partial_+ S(M)} h(x,\xi) I^0 \mH_{\perp} f(x,\xi) \mu \,d(\partial S(M))
\end{align*}
which is the required result.

To prove \eqref{hperp_integrationbyparts} it is enough to show that 
\begin{equation*}
\int_{SM} \mH_{\perp} v \,d(SM) = 0
\end{equation*}
for any $v \in C^{\infty}(SM)$ with $v = 0$ near $\partial M$. Using the isometry $F: SM \to SM$, $(x,\xi) \mapsto (x,-\xi_{\perp})$ and the invariance of $\nabla$, a change of variables and Santal{\'o}'s formula imply that 
\begin{align*}
 & \int_{SM} \mH_{\perp} v \,d(SM) = \int_{SM} (\xi_{\perp})^j \nabla_j v \,d(SM) \\
 &= \int_{SM} \xi^j F^* (\nabla_j v) \,d(SM) = \int_{SM} \xi^j \nabla_j (F^* v) \,d(SM) \\
 &= \int_{SM} \mH(F^* v) \,d(SM) \\
 &= \int_{\partial_+ S(M)} \int_0^{\tau(x,\xi)} \mH(F^* v)(\varphi_t(x,\xi)) \mu \,dt \,d(\partial S(M)).
\end{align*}
Since $\mH(F^* v)(\varphi_t(x,\xi)) = \frac{\partial}{\partial t} [(F^* v)(\varphi_t(x,\xi))]$ and $F^* v$ vanishes near $\partial M$, the last integral vanishes.
\end{proof}

Note that $\mH_{\perp} f(x,\xi) = (\xi_{\perp})^j \frac{\partial f}{\partial x_j}(x) = *df(\xi)$, where we write $\alpha(\xi) = \alpha_j \xi^j$ for a $1$-form $\alpha$ and tangent vector $\xi$, and $*$ is the Hodge star operator. Thus the formal adjoint of $S$ is given by  
\begin{equation*}
S^*: C^{\infty}_c(M^{\text{int}}) \to C^{\infty}_{\alpha}(\partial_+ S(M)), \ S^* f = -\frac{1}{2\pi} I^0(*df(\xi)).
\end{equation*}
The injectivity of $S^*$ follows immediately: if $S^* f \equiv 0$ then the ray transform of the solenoidal 1-form $*df$ vanishes, which implies that $f \equiv 0$ by \cite{An}. To prove surjectivity of $S$ it would then be enough to show that $S^*$ has closed range in proper spaces. The actual proof of the surjectivity will be slightly different, and we will proceed as in \cite[Theorem 4.3]{pestovuhlmann_imrn}.

\begin{proof}[Proof of Lemma \ref{lemma:s_surjective}]
We may assume that $(M,g)$ is embedded in a compact surface $(N,g)$ without boundary, and that there is a finite open cover $\{ U_1, \ldots, U_k \}$ of $N$ such that $M \subseteq U_1$, $M \cap \overline{U}_j = \emptyset$ for $j \geq 2$, and each $(\overline{U}_j,g)$ is simple. Let $\varphi_j \in C^{\infty}_c(U_j)$ be a partition of unity so that $\varphi_j \geq 0$, $\sum_{j=1}^k \varphi_j^2 = 1$ in $N$, and $\varphi_1 = 1$ near $M$. Let also $I_j$ be the geodesic ray transform (with zero attenuation) on $1$-forms in $(\closure{U}_j,g)$. Define the operator acting on smooth $1$-forms on $N$, 
\begin{equation*}
P: C^{\infty}(N,\Lambda^1) \to C^{\infty}(N,\Lambda^1), \ \ P \alpha = \sum_{j=1}^k \varphi_j I_j^* I_j (\varphi_j \alpha).
\end{equation*}
For the following details see \cite[Theorem 4.3]{pestovuhlmann_imrn}. The principal symbol of $P$ is given by the following expression, where $c_2$ is a constant, 
\begin{equation*}
\sigma(P)_i^j = c_2(\delta_i^j/\abs{\xi} - \xi^j \xi_i /\abs{\xi}^3).
\end{equation*}
If $-\Delta$ is the Laplace-Beltrami operator on $N$ mapping smooth functions to smooth functions whose integral over $N$ vanishes, define the operator 
\begin{equation*}
\Lambda:  C^{\infty}(N,\Lambda^1) \to C^{\infty}(N,\Lambda^1), \ \ \Lambda \alpha  = -c_2 d (-\Delta)^{-3/2} \delta \alpha.
\end{equation*}
Then $P + \Lambda$ is an elliptic pseudodifferential operator of order $-1$ acting on $1$-forms in $N$.

Consider the operator 
\begin{equation} \label{t_def}
T: C^{\infty}(N,\Lambda^1) \to C^{\infty}(M), \ \ T \alpha = r_M S_1 I_1 (\varphi_1 \alpha).
\end{equation}
Here $r_M$ is the restriction to $M$, and $S_1$ is the operator \eqref{s_def} in $(\closure{U}_1,g)$. We wish to show that $T$ is surjective. Assuming this we can easily finish the proof of the lemma: If $f \in C^{\infty}(M)$ is given, there is a smooth $1$-form $\alpha$ in $N$ with $r_M S_1 I_1 (\varphi_1 \alpha) = f$. Define 
\begin{equation*}
h(x,\xi) = (I_1 (\varphi_1 \alpha))_{\psi_1}(x,\xi), \qquad (x,\xi) \in \partial_+ S(M),
\end{equation*}
where $w_{\psi_1}$ is the map \eqref{w_def} in $(\closure{U}_1,g)$. Since $M$ is strictly contained in $U_1$ we have $h \in C^{\infty}_{\alpha}(\partial_+ S(M))$, and it follows that $Sh = r_M S_1 I_1 (\varphi_1 \alpha) = f$ as required.

It remains to prove surjectivity of \eqref{t_def}. In order to do this we express $T$ as the adjoint of an operator involving $P$. The dual of $C^{\infty}(M)$ may be identified with the set $\mDp_M(N) = \{ v \in \mDp(N) \,;\, \supp(v) \subseteq M \}$. Given $v \in \mDp_M(N)$ choose $v_j \in C^{\infty}(N)$ with $v_j \to v$ in $\mDp(N)$ and $\supp(v_j) \subseteq \{ \varphi_1 = 1 \}$. Then for $\alpha \in C^{\infty}(N,\Lambda^1)$, in the dual pairing in the indicated manifolds, by Lemma \ref{lemma:s_adjoint} we have 
\begin{align*}
(T\alpha, v)_M &= \lim \ (\varphi_1 S_1 I_1 (\varphi_1 \alpha), v_j)_{\closure{U}_1} \\
 &= -\frac{1}{2\pi} \lim \ (I_1(\varphi_1 \alpha), I_1(*d(\varphi_1 v_j)))_{L^2_{\mu}(\partial_+ S(\closure{U}_1))} \\
 &= -\frac{1}{2\pi} \lim \ (\alpha, P(*d v_j))_N = -\frac{1}{2\pi} (\alpha, P(*dv))_N
\end{align*}
since $P$ is continuous on $\mDp(N)$. Now $\Lambda(*dv) = 0$, so we consider 
\begin{equation*}
Q: \mDp(N) \to \mDp(N,\Lambda^1), \ \ Qv = -\frac{1}{2\pi} (P+\Lambda)(*dv).
\end{equation*}
It follows that $T = r_M Q^*$, since for $\alpha \in C^{\infty}(N,\Lambda^1)$ and $v \in C^{\infty}_c(M^{\text{int}})$ 
\begin{equation*}
(T\alpha,v)_M = (\alpha, Qv)_N = (Q^* \alpha, v)_N = (r_M Q^* \alpha,v)_M.
\end{equation*}
Thus $T$ in \eqref{t_def} is a continuous linear map between Fr\'echet spaces. Its adjoint is given by 
\begin{equation*}
T^* = Q|_{\mDp_M(N)}: \mDp_M(N) \to \mDp(N,\Lambda^1).
\end{equation*}
The operator $T^*$ is injective. To see this, let  $v \in \mDp_M(N)$ satisfy $T^* v = 0$. Then also $(P+\Lambda)(*dv) = 0$, showing that $dv$ is smooth by elliptic regularity for $P+\Lambda$. Thus $v$ is smooth and $\varphi_1 I_1^* I_1 (\varphi_1 * dv) = 0$. Consequently $I_1(\varphi_1 *dv) = 0$ in $\partial_+ S(\closure{U}_1)$, and the uniqueness result for $I_1$ \cite{An} implies that $*dv = dp$ near $M$ for some smooth function $p$. Since $v$ is smooth and supported in $M$, this is only possible if $v = 0$.

Using that $Q$ is an elliptic pseudodifferential operator, a standard argument (see for instance \cite[proof of Theorem 6.3.1]{DH2}) shows that the range of $T^*$ is weak* closed in $\mDp(N,\Lambda^1)$. It follows that $T$ is a continuous linear map between Fr\'echet spaces, and $T^*$ is injective with weak* closed range. Then $T$ is surjective by \cite[Theorem 37.2]{treves}, which concludes the proof.
\end{proof}

\section{Holomorphic solutions and uniqueness} \label{sec:uniqueness}

In this section we will prove the uniqueness result establishing injectivity of the attenuated ray transform. By using Proposition \ref{prop:holomorphic_integratingfactor}, the boundary problem 
\begin{equation*}
(\mH + a)u = -f \text{ in } SM, \quad u|_{\partial S(M)} = 0
\end{equation*}
is equivalent with the problem 
\begin{equation*}
\mH(e^{-w} u) = -e^{-w} f \text{ in } SM, \quad e^{-w} u|_{\partial S(M)} = 0.
\end{equation*}
Here $w$ is holomorphic, so also the right hand side $e^{-w} f$ is holomorphic.

The next main ingredient in the uniqueness proof is the following result, showing that in this situation also the solution $e^{-w} u$ is necessarily holomorphic.

\begin{prop} \label{prop:holomorphicsolution_functions}
Let $(M,g)$ be a 2D simple manifold, and let $\tilde{f}$ be a smooth function on $SM$ holomorphic (resp.~antiholomorphic) in the angular variable. Suppose that $v$ is a smooth function on $SM$ satisfying 
\begin{equation*}
\mH v = - \tilde{f} \text{ in $SM$}, \quad v|_{\partial S(M)} = 0.
\end{equation*}
Then $v$ is holomorphic (resp.~antiholomorphic) in the angular variable, and $v_0 = 0$.
\end{prop}
\begin{proof}
We only consider the holomorphic case. It is enough to show that $(\id - iH) v = 0$. The commutator identity \eqref{commutator_formula} shows that 
\begin{align*}
\mH(\id - iH)v &= (\id-iH) \mH v +i[H,\mH] v \\
 &= -(\id-iH) \tilde{f} + i \mH_{\perp} v_0 + i(\mH_{\perp} v)_0.
\end{align*}
We have $(\id - iH)\tilde{f} = \tilde{f}_0$ since $\tilde{f}$ is holomorphic. Now 
\begin{equation} \label{vminus_holomorphic_eq}
\mH(\id - iH)v = -h - \alpha_j \xi^j
\end{equation}
where $h = \tilde{f}_0 - i(\mH_{\perp} v)_0$ is a function and $\alpha = *d (-i v_0)$ is a solenoidal $1$-form ($*$ is the Hodge star operator).

What is important here is that there is no $\xi$-dependence in $h$ and $\alpha$. In fact, since $(\id-iH)v$ vanishes on $\partial S(M)$, \eqref{vminus_holomorphic_eq} implies that the unattenuated ray transform of $h + \alpha_j \xi^j$ vanishes identically. Uniqueness in the unattenuated case \cite{An}, using that $\alpha$ is solenoidal, implies that $h = 0$ and $\alpha = 0$. We have proved that 
\begin{equation*}
\mH(\id - iH)v = 0,
\end{equation*}
and $(\id-iH)v = 0$ since $v|_{\partial S(M)} = 0$.
\end{proof}

It is now easy to give a proof of the injectivity result for the attenuated ray transform of functions.

\begin{proof}[Proof of Theorem \ref{thm:injectivity_functions}]
Let first $f \in C^{\infty}_c(M^{\text{int}})$, and assume that $I^a f \equiv 0$. Then the function 
\begin{equation*}
u(x,\xi) = \int_0^{\tau(x,\xi)} f(\gamma(t,x,\xi)) e^{\int_0^t a(\gamma(s,x,\xi)) \,ds} \,dt
\end{equation*}
is smooth in $SM$ and satisfies 
\begin{equation*}
(\mH + a)u = -f \ \text{ in } SM, \quad u|_{\partial S(M)} = 0.
\end{equation*}
By Proposition \ref{prop:holomorphic_integratingfactor} there is a holomorphic $w \in C^{\infty}(SM)$ and an antiholomorphic $\tilde{w} \in C^{\infty}(SM)$ such that 
\begin{align*}
\mH(e^{-w} u) &= -e^{-w} f, \\
\mH(e^{-\tilde{w}} u) &= -e^{-\tilde{w}} f.
\end{align*}

Now Proposition \ref{prop:holomorphicsolution_functions} shows that $e^{-w} u$ is holomorphic and $e^{-\tilde{w}} u$ is antiholomorphic. Multiplying by $e^w$ and $e^{\tilde{w}}$, it follows that $u$ itself is both holomorphic and antiholomorphic. This is only possible when $u$ does not depend on $\xi$, that is, $u \equiv u_0$. Now the transport equation reads 
\begin{equation*}
du_0(\xi) + au_0 = -f \ \text{ in } SM, \quad u_0|_{\partial S(M)} = 0.
\end{equation*}
Evaluating this at $\pm \xi$ and substracting the resulting expressions gives that $du_0 \equiv 0$, hence $u_0 \equiv 0$ by the boundary condition. Consequently $f \equiv 0$.

It remains to prove the result when $f \in C^{\infty}(M)$ may have support extending up to the boundary. In fact, this case can be reduced to the result for compactly supported functions by using the general principle that $(I^a)^* I^a$ is an elliptic pseudodifferential operator.

Suppose $f \in C^{\infty}(M)$ and $I^a f \equiv 0$. We consider more generally the weighted ray transform with weight $\rho \in C^{\infty}(SM)$, 
\begin{equation*}
I_{\rho} f(x,\xi) = \int_0^{\tau(x,\xi)} \rho(\varphi_t(x,\xi)) f(\gamma(t,x,\xi)) \,dt, \quad (x,\xi) \in \partial_+ S(M).
\end{equation*}
With the choice $\rho = e^{-u^a_-}$, we obtain $I_{\rho} f \equiv 0$. Let $(\tilde{M},g) \supset \supset (M,g)$ be another simple manifold which is so small that any $\tilde{M}$-geodesic starting at a point of $\partial_- S(M)$ never enters $M$ again. We extend $a$ to $\tilde{M}$ as a smooth function and $f$ by zero to $\tilde{M}$, and denote by $\tilde{I}_{\rho}$ the corresponding weighted ray transform in $\tilde{M}$.

It is easy to see that $\tilde{I}_{\rho} f(y,\eta) = 0$ for all $(y,\eta) \in \partial_+ S(\tilde{M})$, since either the geodesic starting from $(y,\eta)$ never touches $M$ or else $\tilde{I}_{\rho} f(y,\eta) = I_{\rho} f(x,\xi)$ for some $(x,\xi) \in \partial_+ S(M)$. By \cite[Proposition 2]{FSU}, since $\rho$ is nonvanishing, $\tilde{I}_{\rho}^* \tilde{I}_{\rho}$ is an elliptic pseudodifferential operator of order $-1$ in $\tilde{M}^{\text{int}}$. Now $\tilde{I}_{\rho}^* \tilde{I}_{\rho} f = 0$, and elliptic regularity shows that $f$ is smooth. Thus $f \in C^{\infty}_c(\tilde{M}^{\text{int}})$ and $\tilde{I}_{\rho} f \equiv 0$, showing that $\tilde{I}^a f \equiv 0$. The result above implies that $f \equiv 0$ as required.
\end{proof}

We proceed to the attenuated ray transform of $1$-forms. In this case the required generalization of Proposition \ref{prop:holomorphicsolution_functions} is as follows.

\begin{prop} \label{prop:holomorphicsolution_oneforms}
Let $(M,g)$ be a 2D simple manifold and let $\rho \in C^{\infty}(SM)$ be holomorphic (resp.~antiholomorphic) in the angular variable. Suppose $F(x,\xi) = f(x) + \alpha_j(x) \xi^j$ where $f$ is a smooth function and $\alpha$ is a smooth $1$-form on $M$. If $v \in C^{\infty}(SM)$ satisfies 
\begin{equation*}
\mH v = - \rho F \ \text{ in } SM, \quad v|_{\partial S(M)} = 0,
\end{equation*}
then $v$ is holomorphic (resp.~antiholomorphic) in the angular variable.
\end{prop}
\begin{proof}
We only prove the holomorphic case. As in Proposition \ref{prop:holomorphicsolution_functions}, we study the function $(\id-iH)v$ and note that 
\begin{align}
\mH(\id - iH)v &= (\id-iH) \mH v +i[H,\mH] v \notag \\
 &= -(\id-iH) (\rho F) + i \mH_{\perp} v_0 + i(\mH_{\perp} v)_0. \label{antiholomorphic_computation_oneform}
\end{align}
We would like to show that $(\id-iH)(\rho F)$ is a first order polynomial in $\xi$.

The first step is to prove that 
\begin{equation} \label{hilbert_rho_rho_0}
(\id-iH)((\rho-\rho_0) F) = \tilde{h}
\end{equation}
for some $\tilde{h} = \tilde{h}(x)$. This can be reduced to elementary facts about Fourier coefficients. Fix a point $x$ in $M$, and choose an orientation preserving isometry $\Phi: S_x \to S^1$. If $H_0$ is the Hilbert transform on $S^1$, it is enough to prove the equivalent statement 
\begin{equation} \label{hilbert_s1_equivalent}
(\id-iH_0)((\Phi^{-1})^* (\rho-\rho_0)F) = \tilde{h}(x).
\end{equation}
But $(\id-iH_0)((\Phi^{-1})^* (\rho-\rho_0)) = 0$ since $\rho$ is holomorphic, meaning that one has the Fourier series 
\begin{equation*}
(\Phi^{-1})^* (\rho-\rho_0) = \sum_{k=1}^{\infty} a_k e^{ik\theta}.
\end{equation*}
Also, since $F$ is a first order polynomial in $\xi$, one has the Fourier series $(\Phi^{-1})^* F = b_0 + b_1 e^{i\theta} + b_{-1} e^{-i\theta}$. Multiplying these Fourier series gives \eqref{hilbert_s1_equivalent}.

By \eqref{antiholomorphic_computation_oneform} and \eqref{hilbert_rho_rho_0}, we obtain 
\begin{equation*}
\mH(\id - iH)v = -\rho_0(\id-iH)F - \tilde{F}
\end{equation*}
where $\tilde{F} = \tilde{h} - i(\mH_{\perp} v)_0 - i*dv_0(\xi)$ is a first order polynomial in $\xi$. Next we employ a Hodge decomposition $\alpha = dp + *dq$ where $p,q \in C^{\infty}(M)$. One has the antiholomorphic projections 
\begin{align*}
(\id-iH)(dp(\xi)) &= (\id-iH)\mH p = \mH(\id-iH)p -i[H,\mH]p \\
 &= (\mH - i \mH_{\perp})p
\end{align*}
and similarly 
\begin{align*}
(\id-iH)(*dq(\xi)) &= (\id-iH)\mH_{\perp} q = (\id-iH) [H,\mH] q = (\id-iH)H \mH q \\
 &= H \mH q + i (\mH q -  (\mH q)_0) = i(\id-iH)\mH q \\
 &= i(\mH - i \mH_{\perp})q,
\end{align*}
the last line using the above computation for $p$.

Putting these results together, we have proved that 
\begin{equation*}
\mH(\id-iH)v = -\hat{F} \,\text{ in } SM, \quad (\id-iH)v|_{\partial S(M)} = 0,
\end{equation*}
where $\hat{F}$ is the first order polynomial in $\xi$ given by 
\begin{equation*}
\hat{F} = \rho_0 f + \tilde{h} - i(\mH_{\perp} v)_0 + \rho_0(d-i*d)(p+iq)(\xi) - i*d v_0(\xi).
\end{equation*}
This shows that the geodesic ray transform of $\hat{F}$ vanishes. Therefore $\hat{F}(\xi) = d\hat{p}(\xi)$ for some $\hat{p} \in C^{\infty}(M)$ with $\hat{p}|_{\partial M} = 0$, and the equation for $(\id-iH)v$ implies that 
\begin{equation*}
(\id-iH)v = -\hat{p}.
\end{equation*}
This proves that $v$ is holomorphic.
\end{proof}

The injectivity result for $1$-forms is now proved in a similar way as the corresponding result for functions.

\begin{proof}[Proof of Theorem \ref{thm:injectivity_oneforms}]
Let $f$ be a smooth function and $\alpha$ a smooth $1$-form on $M$, and let $F = f + \alpha_j \xi^j$. Assume first that $f$ and $\alpha$ are compactly supported in $M^{\text{int}}$. Then the function 
\begin{equation*}
u(x,\xi) = \int_0^{\tau(x,\xi)} F(\varphi_t(x,\xi)) e^{\int_0^t a(\gamma(s,x,\xi)) \,ds} \,dt
\end{equation*}
is smooth in $SM$ and satisfies the equation 
\begin{equation*}
(\mH + a)u = -F \,\text{ in } SM, \quad u|_{\partial S(M)} = 0.
\end{equation*}
As before we use Proposition \ref{prop:holomorphic_integratingfactor} to find a holomorphic function $w$ and an antiholomorphic function $\tilde{w}$ such that 
\begin{align*}
\mH(e^{-w} u) &= -e^{-w} F, \\
\mH(e^{-\tilde{w}} u) &= -e^{-\tilde{w}} F.
\end{align*} 
Proposition \ref{prop:holomorphicsolution_oneforms} implies that $e^{-w} u$ is holomorphic and $e^{-\tilde{w}} u$ is antiholomorphic, and consequently $u$ is both holomorphic and antiholomorphic. Thus $u \equiv u_0$, and the transport equation can be written as 
\begin{equation*}
(du_0 + \alpha)(\xi) + au_0 = -f \,\text{ in } SM, \quad u_0|_{\partial S(M)} = 0.
\end{equation*}
Consequently $\alpha = -du_0$ and $f = -au_0$ which is the required result.

Again the general case where $f \in C^{\infty}(M)$ and $\alpha$ is a smooth $1$-form in $M$ may be reduced to the previous case by elliptic regularity. Assume that $I^a(f+\alpha_j \xi^j) \equiv 0$. If $\rho = e^{-u^a_-}$, this implies that $I_{\rho} (f + \alpha_j \xi^j) \equiv 0$ where 
\begin{equation*}
I_{\rho}F(x,\xi) = \int_0^{\tau(x,\xi)} \rho(\varphi_t(x,\xi)) F(\varphi_t(x,\xi)) \,dt.
\end{equation*}
Consider the solenoidal decomposition $\alpha = \alpha^s + dp$, where $\delta \alpha^s = 0$ and $-\Delta p = \delta \alpha$ with $p|_{\partial M} = 0$. An integration by parts shows that we have $I_{\rho}(f-ap + \alpha_j^s \xi^j) \equiv 0$.

Let $(\tilde{M},g) \supset \supset (M,g)$ be a simple manifold as in the proof of Theorem \ref{thm:injectivity_functions}, and extend $a$ smoothly to $\tilde{M}$ and $f, p, \alpha^s$ by zero to $\tilde{M}$. It follows that $\tilde{I}_{\rho} (f - ap + \alpha^s_j \xi^j) \equiv 0$. By \cite[Proposition 1]{HS} (where we make the choices $w = \rho$ and $\alpha = \rho$, so that the modified elliptic condition in \cite[Remark 1]{HS} is satisfied), $\tilde{I}_{\rho}^* \tilde{I}_{\rho}$ is a pseudodifferential operator of order $-1$ in $\tilde{M}^{\text{int}}$ which is elliptic in the sense that whenever $f', \alpha'$ are in $L^2(\tilde{M}^{\text{int}})$ and $\tilde{I}_{\rho}^* \tilde{I}_{\rho} (f' + \alpha_j' \xi^j) \equiv 0$ and $\delta \alpha' \equiv 0$, then $f'$ and $\alpha'$ are smooth. Thus, $f-ap$ and $\alpha^s$ are smooth and compactly supported in $\tilde{M}^{\text{int}}$ and $\tilde{I}_{\rho}(f-ap + \alpha^s_j \xi^j) \equiv 0$, showing that $f-ap = a\tilde{p}$ and $\alpha^s = d\tilde{p}$ for some smooth $\tilde{p}$ with $\tilde{p}|_{\partial \tilde{M}} = 0$. Since $\alpha^s$ is zero outside $M$ it follows that $\tilde{p}$ vanishes outside $M$, and one obtains $f = a(p+\tilde{p})$ and $\alpha = d(p+\tilde{p})$ in $M$ with $p+\tilde{p}$ vanishing on $\partial M$.
\end{proof}

Finally we give a proof of the stability result in the introduction.

\begin{proof}[Proof of Theorem \ref{thm:stability}]
We make the same preparations as in the end of proof of Theorem \ref{thm:injectivity_oneforms}. Thus, we consider the decomposition $\alpha = \alpha^s + dp$ with $p = G(\delta \alpha)$ in $M$, choose a slightly larger manifold $(\tilde{M},g)$, and extend $a$ to $\tilde{M}$ and $f$,$p$,$\alpha^s$ by zero to $\tilde{M}$. Let $\tilde{N} = \tilde{I}_{\rho}^* \tilde{I}_{\rho}$ with $\rho = e^{-u^a_-}$. Since the modified elliptic condition of \cite[Remark 1]{HS} is satisfied, \cite[Proposition 2]{HS} implies the estimate 
\begin{multline} \label{stability_estimate_first}
\norm{f-ap}_{L^2(M)} + \norm{\alpha^s}_{L^2(M)} \leq C(\norm{\tilde{N}(f-ap + \alpha^s_j \xi^j)}_{H^1(\tilde{M})} \\
 + \norm{f-ap}_{H^{-1}(\tilde{M})} + \norm{\alpha^s}_{H^{-1}(\tilde{M})}).
\end{multline}

Let $L^2_s(M) = \{ \phi+\beta_j \xi^j \,;\, \phi \in L^2(M), \beta \text{ is a $1$-form in $L^2(M)$ and $\delta \beta = 0$} \}$ be the space of solenoidal pairs, assumed to be extended by zero to $\tilde{M}$. As in the end of proof of Theorem \ref{thm:injectivity_oneforms}, one can use the ellipticity of $\tilde{N}$ and injectivity of $\tilde{I}_{\rho}$ to see that $\tilde{N}: L^2_s(M) \to H^1(\tilde{M})$ is a bounded injective operator. We can then use \cite[Lemma 1]{HS} with the choices $X = L^2_s(M)$, $Y = H^1(\tilde{M})$, and $Z = H^{-1}(\tilde{M})$ (the latter two being the natural Sobolev spaces for solenoidal pairs), and conclude from \eqref{stability_estimate_first} that 
\begin{equation*}
\norm{f-ap}_{L^2(M)} + \norm{\alpha^s}_{L^2(M)} \leq \tilde{C} \norm{\tilde{N}(f-ap + \alpha^s_j \xi^j)}_{H^1(\tilde{M})}.
\end{equation*}
The stability result follows by noting that $\tilde{N} (f+\alpha_j \xi^j) = \tilde{N}(f-ap + \alpha^s_j \xi^j)$ and by taking $M_1 = \tilde{M}$.
\end{proof}

\section{Reconstruction procedure} \label{sec:reconstruction}

Let $(M,g)$ be a simple 2D manifold and let $a$ be a smooth complex function on $M$. In this section we give a procedure for determining a smooth function $f$ in $M$ from the knowledge of $I^a f$.

There are two nontrivial parts in the procedure: computing the inverse of the unattenuated ray transform $I^0$ in $(M,g)$, and the construction of (anti)holomorphic integrating factors for the equation $(\mH+a)u = 0$. If $(M,g)$ has constant curvature these operations can be done explicitly (since the $W$ operator vanishes, see \cite{pestovuhlmann_imrn} and Corollary \ref{cor:integratingfactors}). Also, if $(M,g)$ is a small perturbation of a constant curvature manifold then the $W$ operator has small norm on $L^2(M)$ and these two operations can be expressed in terms of convergent Neumann series, see \cite{krishnan} and Corollary \ref{cor:integratingfactors} again. However, for general simple $(M,g)$ it is not clear how to carry out these operations in an explicit way.

For simplicity we will assume below that $f \in C^{\infty}_c(M^{\text{int}})$, since in this case all the functions will be smooth up to $\partial M$ and we do not need to worry about regularity issues. Theorem \ref{thm:reconstruction} immediately follows from the next result.

\begin{prop}
Under the stated assumptions, a function $f \in C^{\infty}_c(M^{\text{int}})$ can be determined from the knowledge of $I^a f$ using the following procedure:

\begin{enumerate}

\item[1.] 
Define a function $d$ on $\partial S(M)$ by 
\begin{equation*}
d(x,\xi) = \left\{ \begin{array}{cl} I^a f(x,\xi), & \quad (x,\xi) \in \partial_+ S(M), \\ 0, & \quad \text{otherwise}. \end{array} \right.
\end{equation*}

\item[2.] 
Find a holomorphic function $w$ and an antiholomorphic function $\tilde{w}$, both smooth odd functions on $SM$, such that $\mH w = \mH \tilde{w} = -a$.

\item[3.] 
Let $\beta = (\id-iH)(e^{-w} d)$ and $\tilde{\beta} = (\id+iH)(e^{-\tilde{w}} d)$ on $\partial S(M)$.

\item[4.] 
Let $v = \beta \circ \psi + u^{(I^0)^{-1}(A_-^* \beta)}$ and $\tilde{v} = \tilde{\beta} \circ \psi + u^{(I^0)^{-1}(A_-^* \tilde{\beta})}$ in $SM$, where $A_-^* \beta = \beta - \beta \circ \psi$ on $\partial_+ S(M)$ and $(I^0)^{-1}$ is the inverse of the geodesic ray transform in $(M,g)$, in the sense that 
\begin{equation*}
(I^0)^{-1} I^0 (\phi + \alpha_j \xi^j) = \phi + \alpha_j \xi^j
\end{equation*}
for a smooth function $\phi$ and a solenoidal $1$-form $\alpha$.

\item[5.] 
Define $\hat{m} = \frac{1}{4} (\id-iH)(e^w v) + \frac{1}{4} (\id+iH)(e^{\tilde{w}} \tilde{v})$ and $\hat{u} = \hat{m} - \hat{m}_0$.

\item[6.] 
Define $q = (d-\hat{u})_0 \circ \psi + (u^{(\mH \hat{u} + a \hat{u})_-})_0$, and let $u = q + \hat{u}$.

\item[7.] 
Let $f = -(\mH u + au)_0$.

\end{enumerate}

\end{prop}
\begin{proof}
Let $u$ be the solution of $\mH u + au = -f$ in $SM$ with $u|_{\partial_- S(M)} = 0$, so that $u|_{\partial_+ S(M)} = I^a f$. If $w$ and $\tilde{w}$ are as described, then one has the two equations 
\begin{align*}
\mH(e^{-w} u) &= -e^{-w} f, \\
\mH(e^{-\tilde{w}} u) &= -e^{-\tilde{w}} f.
\end{align*}
The right hand side in the first equation is holomorphic in the angular variable. We will show that $v = (\id-iH)(e^{-w} u)$, the antiholomorphic part of the solution $e^{-w} u$, is determined by $I^a f$. In fact, the computation in Proposition \ref{prop:holomorphicsolution_functions} shows that 
\begin{equation*}
\mH v = -\phi - \alpha_j \xi^j
\end{equation*}
where $\phi = f - i (\mH_{\perp} (e^{-w} u))_0$ and $\alpha = *d(-i(e^{-w} u)_0)$. Here we used that $(e^{-w})_0 = 1$ since $w$ is odd. We have that $v - v \circ \psi|_{\partial_+ S(M)} = I^0 (\phi + \alpha_j \xi^j)$. Here $\alpha$ is solenoidal, so $I^0$ is invertible and 
\begin{equation*}
\phi + \alpha_j \xi^j = (I^0)^{-1}(v - v \circ \psi|_{\partial_+ S(M)}).
\end{equation*}
This proves that $v$ is the function given in Step 4 above. A similar argument shows that $\tilde{v} = (\id+iH)(e^{-\tilde{w}} u)$ is the other function in Step 4.

We have obtained two decompositions 
\begin{align*}
e^{-w} u &= h + \frac{1}{2} v, \\
e^{-\tilde{w}} u &= \tilde{h} + \frac{1}{2} \tilde{v}
\end{align*}
where $h$ is holomorphic, $\tilde{h}$ is antiholomorphic, and $v$ and $\tilde{v}$ can be determined from the attenuated ray transform of $f$. This results in two decompositions for the solution $u$, 
\begin{align*}
u &= e^w h + \frac{1}{2} e^w v, \\
u &= e^{\tilde{w}} \tilde{h} + \frac{1}{2} e^{\tilde{w}} \tilde{v},
\end{align*}
where again $e^w h$ is holomorphic, $e^{\tilde{w}} \tilde{h}$ is antiholomorphic, and $e^w v$ and $e^{\tilde{w}} \tilde{v}$ are known. This determines $u$ up to a term which is constant in $\xi$, which can be seen by writing 
\begin{equation*}
u = \frac{1}{2}(\id+iH) u + \frac{1}{2}(\id-iH) u = \frac{1}{2} (e^w h)_0 + \frac{1}{2} (e^{\tilde{w}} \tilde{h})_0 + \hat{m}
\end{equation*}
with $\hat{m}$ given in Step 5.

Write $u = q + \hat{u}$ where $\hat{u}_0 = 0$. Then $q = u_0$ and $\hat{u} = \hat{m} - \hat{m}_0$. To find the term $q$ we note that $q|_{\partial M} = u-\hat{u}|_{\partial M} = (d-\hat{u})_0|_{\partial M}$, where necessarily $u-\hat{u}$ is independent of $\xi$. Taking the odd part in the equation $\mH u + au = -f$ implies that 
\begin{equation*}
\mH q = - (\mH \hat{u} + a\hat{u})_-.
\end{equation*}
Therefore $q$ is given by the quantity in Step 6. We have determined the solution $u$ in $SM$ from the knowledge of $I^a f$. Now $f = -\mH u - au$, and taking averages proves the formula in Step 7.
\end{proof}

\providecommand{\bysame}{\leavevmode\hbox to3em{\hrulefill}\thinspace}
\providecommand{\href}[2]{#2}

\end{document}